\documentclass[a4paper,12pt]{article}
\usepackage{latexsym,amsmath,amsthm,amssymb}
\usepackage{a4wide}
\usepackage{hyperref}
\usepackage{marginnote}
\usepackage{color}
\usepackage{cite}
\hypersetup{
pdftitle={mean zero n dim}   
pdfauthor={Van Hoang Nguyen},
colorlinks = true,
linkcolor = magenta,
citecolor = blue,
}

\theoremstyle{plain}
\newtheorem{theorem}{Theorem}[section]

\newtheorem{proposition}[theorem]{Proposition}
\newtheorem{lemma}[theorem]{Lemma}

\theoremstyle{definition}

\theoremstyle{remark}

\renewcommand{\thefootnote}{\arabic{footnote}}

\def\R{\mathbb R}

\def\al{\alpha}
\def\om{\omega}
\def\Om{\Omega}
\def\be{\beta}
\def\de{\delta}
\def\De{\Delta} 

\def\lam{\lambda}
\def\Lam{\Lambda}
\def\vphi{\varphi}
\def\ep{\epsilon}
\def\na{\nabla}
\def\pa{\partial}
\def\lt{\left}
\def\rt{\right}
\def\o{\overline}

\def\i0i{\int_0^\infty}

\numberwithin{equation}{section}

\title{Improved Moser--Trudinger inequality for functions with mean value zero in $\R^n$ and its extremal functions}
\author{Van Hoang Nguyen\footnote{
Institut de Math\'ematiques de Toulouse, Universit\'e Paul Sabatier, 118 Route de Narbonne, 31062 Toulouse c\'edex 09, France.}
}

\begin{document}
\maketitle


\renewcommand{\thefootnote}{}

\footnote{Email: \href{mailto: Van Hoang Nguyen <van-hoang.nguyen@math.univ-toulouse.fr>}{van-hoang.nguyen@math.univ-toulouse.fr}}

\footnote{2010 \emph{Mathematics Subject Classification\text}: 46E35, 26D10.}

\footnote{\emph{Key words and phrases\text}: Moser--Trudinger inequality, blow-up analysis, sharp constant, extremal functions, elliptic estimates.}

\renewcommand{\thefootnote}{\arabic{footnote}}
\setcounter{footnote}{0}

\begin{abstract}
Let $\Omega$ be a bounded smooth domain in $\mathbb R^n$, $W^{1,n}(\Omega)$ be the Sobolev space on $\Omega$, and $\lambda(\Omega) = \inf\{\|\nabla u\|_n^n: \int_\Omega u dx =0, \|u\|_n =1\}$ be the first nonzero Neumann eigenvalue of the $n-$Laplace operator $-\Delta_n$ on $\Om$. For $0 \leq \alpha < \lambda(\Omega)$, let us define $\|u\|_{1,\alpha}^n =\|\nabla u\|_n^n -\alpha \|u\|_n^n$. We prove, in this paper, the following improved Moser--Trudinger inequality on functions with mean value zero on $\Omega$,
\[
\sup_{u\in W^{1,n}(\Omega), \int_\Omega u dx =0, \|u\|_{1,\alpha} =1} \int_{\Omega} e^{\beta_n |u|^{\frac n{n-1}}} dx < \infty,
\]
where $\beta_n = n (\omega_{n-1}/2)^{1/(n-1)}$, and $\omega_{n-1}$ denotes the surface area of unit sphere in $\R^n$. We also show that this supremum is attained by some function $u^*\in W^{1,n}(\Omega)$ such that $\int_\Omega u^* dx =0$ and $\|u^*\|_{1,\alpha} =1$. This generalizes a result of Ngo and Nguyen \cite{NN17} in dimension two and a result of Yang  \cite{Yang07} for $\alpha=0$, and improves a result of Cianchi \cite{Cianchi05}.
\end{abstract}

\section{Introduction}
Let $\Omega$ be a bounded smooth domain in $\R^n$ and $W_0^{1,n}(\Om)$ be completion of $C_0^\infty(\Om)$ under the Dirichlet norm $\|u\|_{W_0^{1,n}(\Om)} =\lt(\int_\Omega |\na u|^n dx\rt)^{1/n}$. The Moser--Trudinger inequality  asserts that
\begin{equation}\label{eq:MT}
\sup_{u\in W_0^{1,n}(\Om), \|\na u\|_n \leq 1} \int_\Om e^{\al u^{\frac n{n-1}}} dx < \infty,
\end{equation}
for any $\alpha \leq \alpha_n := n\omega_{n-1}^{\frac 1{n-1}}$ where $\omega_{n-1}$ denotes the area of unit sphere in $\R^n$. This inequality \eqref{eq:MT} was proved independently by Poho${\rm \check{z}}$aev \cite{P1965}, Yudovi${\rm \check{c}}$ \cite{Y1961} and Trudinger \cite{T1967}. The sharp constant $\alpha_n$ was found by Moser \cite{M1970}.

Let $W^{1,n}(\Om)$ be the completion of $C^\infty(\o \Om)$ under the norm 
\[
\|u\|_{W^{1,n}(\Omega)} = \lt(\|u\|_n^n + \|\na u\|_n^n\rt)^{1/n}.
\]
In \cite{Cianchi05}, Cianchi proved a sharp Moser--Trudinger inequality for functions in $W^{1,n}(\Om)$ with mean value zero as follows
\begin{equation}\label{eq:Cianchi}
\sup_{u\in W^{1,n}(\Om), \int_\Om u dx =0, \|\na u\|_n \leq 1} \int_\Om e^{\beta |u|^{\frac n{n-1}}} dx < \infty,
\end{equation}
for any $\beta \leq \beta_n = n (\om_{n-1}/2)^{1/(n-1)}$. Moreover, if $\beta > \beta_n$ then the supremum in \eqref{eq:Cianchi} will be infinite. In special case when $\Om$ is ball $B^n$ in $\R^n$, the inequality \eqref{eq:Cianchi} was proved by Leckband in \cite{Leckband05}. This inequality generalizes an earlier result of Chang and Yang \cite{CYzero} in dimension two,
\begin{equation}\label{eq:ChangYang}
\sup_{u\in W^{1,2}(\Om), \int_\Om u dx =0, \|\na u\|_2 \leq 1} \int_\Om e^{\beta |u|^{2}} dx < \infty
\end{equation}
for any $\beta \leq 2\pi$.  A sharpened version of \eqref{eq:ChangYang} in spirit of Adimurthi and Druet \cite{AD2004} was proved by Lu and Yang in \cite{LYzero}. 

In \cite{NN17}, Ngo and the author proved another sharpened version of Moser--Trudinger type inequality for functions with mean value zero in dimension two. To state the result in that paper, let us denote by 
\[
\lambda(\Om) = \inf\{\|\na u\|_2^2\,:\, u\in W^{1,2}(\Om), \|u\|_2 =1, \int_\Om u dx =0\}
\]
the first nonzero Neumann eigenvalue of $-\Delta$ on $\Om$, and for $0\leq \alpha < \lam(\Om)$, we denote
\[
\|u\|_{1,\al}^2 = \|\na u\|_2^2 -\alpha \|u\|_2^2.
\]
In \cite{NN17}, Ngo an the author proved the following inequality,
\begin{equation}\label{eq:NN17}
\sup_{u\in W^{1,2}(\Om), \|u\|_{1,\alpha} \leq 1, \int_\Om u dx =0} \int_\Om e^{2\pi u^2} dx < \infty.
\end{equation}
This is an improvement of \eqref{eq:ChangYang} in spirit of Tintarev \cite{Tintarev} for the classical Moser--Trudinger inequality. Such a result recently was proved for the singular Moser--Trudinger inequality in dimension two by Yang and Zhu \cite{YZ}. As shown in \cite{NN17}, \eqref{eq:NN17} is stronger than the one of Lu and Yang \cite{LYzero} and the one of Chang and Yang \eqref{eq:ChangYang}. It is also proved in \cite{NN17} that the supremum in \eqref{eq:NN17} is attained by some functions $u\in W^{1,2}(\Om)$ with $\int_\Om u dx =0$ and $\|u\|_{1,\al} \leq 1$. 

Our goal of this paper is to establish an improvement of type \eqref{eq:NN17} for inequality \eqref{eq:Cianchi}. Let $\Om$ be a smooth bounded domain in $\R^n$, we denote
\[
\mathcal H =\lt\{u\in W^{1,n}(\Om)\,:\, \int_\Om  u dx =0\rt\}
\]
the subspace of $W^{1,n}(\Om)$ consisting the functions of mean value zero. Denote
\[
\lam_1(\Om) = \inf \{\|\na u\|_n^n\, :\, u\in \mathcal H, \|u\|_n =1\}
\]
the first nonzero Neumann eigenvalue of $n-$Laplace $-\De_n$ on $\Om$. By a simple variational argument, we can prove that $\lam_1(\Om)$ is strict positive and is attained by a function in $\mathcal H$. For $0\leq \al < \lam_1(\Om)$, we define
\[
\|u\|_{1,\al}^n = \|\na u\|_n^n -\al \|u\|_n^n,\qquad u\in \mathcal H.
\]
Note that $\|\cdot\|_{n,\al}$ is a norm on $\mathcal H$. Our first main result reads as follows
\begin{theorem}\label{Main1}
Let $\Om$ be a bounded smooth domain in $\R^n$. For any $0\leq \al < \lam_1(\Om)$, it holds
\begin{equation}\label{eq:Mainresult1}
\sup_{u\in \mathcal H, \|u\|_{1,\al} \leq 1} \int_\Om e^{\beta_n |u|^{\frac n{n-1}}} dx < \infty.
\end{equation}
\end{theorem}
Concerning to the existence of maximizers for \eqref{eq:Mainresult1}, we will prove the following result.
\begin{theorem}\label{Main2}
Let $\Om$ be a bounded smooth domain in $\R^n$, and $0\leq \al < \lam_1(\Om)$. There exists $u^* \in \mathcal H$ such that $\|u^*\|_{1,\al} =1$ and
\[
\int_\Om e^{\beta_n {u^*}^{\frac n{n-1}}} dx = \sup_{u\in \mathcal H, \|u\|_{1,\alpha} \leq 1} \int_\Om e^{\beta_n |u|^{\frac n{n-1}}} dx,
\]
i.e., the supremum in \eqref{eq:Mainresult1} is attained.
\end{theorem}
In the case $\alpha =0$, our result \eqref{eq:Mainresult1} reduces to the one of Cianchi \eqref{eq:Cianchi}. In this case, the existence of extremal function for \eqref{eq:Cianchi} was proved by Yang in \cite{Yang07}. As usually, the proof of Theorems \ref{Main1} and \ref{Main2} is based on blow-up analysis. We refer interesting reader to the book \cite{Druet} or articles \cite{AD2004,Li2001,Lin1996,NN17,Yang06,Yang06*,Yang07,Yang09,YZ} for more detail on this technique. We should point out here that, in our situation, the blow-up occurs on the boundary $\pa \Om$ as in \cite{LYzero} which makes more difficult to deal with. The existence of extremal functions for Moser--Trudinger inequality was first proved by Carleson and Chang \cite{CC1986} for unit ball in $\R^n$. This existence result was proved for any smooth domain in $\R^2$ by Flucher \cite{Flucher1992} and then extended to any dimension by Lin \cite{Lin1996}. The existence of extremal functions for Moser--Trudinger inequality on compact Riemannian manifold was studied by Li \cite{Li2005}. For more about the existence of extremal functions for Moser--Trudinger inequality \eqref{eq:MT} and its generalization, we refer reader to \cite{CC1986,Csato2015,Csato2016,Flucher1992,Li2001,Li2005,Lin1996,NN17,Yang06,Yang06*,Yang07,Yang09,YZ} and references therein

The organization of this paper is as follows. In the next section \S2 we prove a subcritical version of \eqref{eq:Mainresult1} and the existence of extremal functions for this subcritical inequality. In section \S3, we analysis asymptotic behavior of the sequence of extremal functions for the subcritical inequality. In section \S4, we establish some capacity estimates which lead to the proof of Theorems \ref{Main1} and \ref{Main2} in section \S5. 

\section{Extremal functions for the subcritical inequalities}
In this section, we study the subcritical Moser--Trudinger inequalities for functions in $\mathcal H$. We will prove the existence of extremal function for these inequalities. For $0< \ep < \beta_n$, we denote $\beta_\ep =\beta_n -\ep$. Let us consider
\[
C_\ep = \sup_{u\in \mathcal H, \|u\|_{1,\al}\leq 1} \int_\Om e^{\beta_\ep |u|^{\frac{n}{n-1}}} dx.
\]
Our result in this section is as follows,
\begin{proposition}\label{subcritical}
Let $\Om$ be a bounded smooth domain in $\R^n$ and $0\leq \al < \lam_1(\Om)$. For any $0< \ep < \beta_n$, we have $C_\ep < \infty$ and that there exists $u_\ep \in \mathcal H \cap C^1(\overline \Om)$ such that $\|u_\ep\|_{1,\al}=1$ and 
\begin{equation}\label{eq:sub}
C_\ep = \int_\Om e^{\beta_\ep |u_\ep|^{\frac n{n-1}}} dx.
\end{equation}
The Euler--Lagrange equation of $u_\ep$ is given by
\begin{equation}\label{eq:EL}
\begin{cases}
-\De_n u_\ep = \frac1{\lam_\ep} e^{\beta_\ep |u_\ep|^{\frac n{n-1}}} |u_\ep|^{\frac{2-n}{n-1}}u_\ep + \al |u_\ep|^{n-2} u_\ep -\frac{\mu_\ep+\al \lam_\ep \nu_\ep}{\lam_\ep}&\mbox{in $\Om$,}\\
\frac{\pa u_\ep}{\pa \nu} = 0&\mbox{on $\pa \Om$,}\\
u_\ep \in \mathcal H,\|u_\ep\|_{1,\al} =1,\mu_\ep = \frac{1}{|\Om|}\int_\Om e^{\beta_\ep |u_\ep|^{\frac n{n-1}}}|u_\ep|^{\frac{2-n}{n-1}} u_\ep dx,\\
\lam_\ep = \int_\Om e^{\beta_\ep |u_\ep|^{\frac n{n-1}}} |u_\ep|^{\frac n{n-1}} dx, \nu_\ep = \frac1{|\Om|} \int_\Om |u_\ep|^{n-2} u_\ep dx,
\end{cases}
\end{equation}
where $\De_n u_\ep = {\rm div}(|\na u_\ep|^{n-2} \na u_\ep)$. Furthermore, it holds
\begin{equation}\label{eq:limitCep}
\lim_{\ep \to 0} C_\ep = \sup_{u\in \mathcal H, \|u\|_{1,\al}\leq 1} \int_\Om e^{\beta_n |u|^{\frac{n}{n-1}}} dx,
\end{equation}
\begin{equation}\label{eq:liminf}
\liminf_{\ep \to 0} \lam_\ep >0,
\end{equation}
and
\begin{equation}\label{eq:upbound}
\frac{|\mu_\ep|}{\lam_\ep} \leq c,\qquad |\nu_\ep|\leq c,
\end{equation}
for some constant $c>0$.
\end{proposition}
In the proof of Proposition \ref{subcritical}, we need  the following Lions type \cite{Lions1985} concentration--compactness principle for functions in $\mathcal H$.
\begin{lemma}\label{Lions}
Let $\{u_j\}_j\subset \mathcal H$ such that $\|u_j\|_{1,\al} =1$ and $u_j \rightharpoonup u_0$ in $W^{1,n}(\Om)$ then for any $0 < p< 1/(1-\|u_0\|_{1,\al}^n)^{1/(n-1)}$, it holds
\[
\limsup_{j\to\infty} \int_\Om e^{\beta_n p |u_j|^{\frac n{n-1}}}dx < \infty.
\]
\end{lemma}
\begin{proof}
Evidently, if $u_0\equiv 0$, then $\|u_j\|_n \to 0$ which implies $\|\na u_j\|_n \to 1$ as $j\to \infty$. Thus, the conclusion follows from \eqref{eq:Cianchi}.

We next consider the case $u_0\not\equiv 0$. By Sobolev embedding, we have
\[
\|\na u_j\|_n^n = 1 + \al \|u_j\|_n^n \to 1 + \al \|u_0\|_n^n.
\]
Denote $v_j = u_j /\|\na u_j\|_n$ then $\|\na v_j\|_n =1$ and 
\[
v_j \rightharpoonup \frac{u_0}{\lt( 1 + \al \|u_0\|_n^n\rt)^{1/n}} =:v_0 \qquad\text{\rm weakly in }\, W^{1,n}(\Om).
\]
By a result of $\check{\rm C}$ern\'y, Cianchi and Hencl \cite{Cerny2013}, we have for any $q < 1/(1-\|\na v_0\|_n^n)^{1/(n-1)}$
\[
\limsup_{j\to\infty} \int_\Om e^{\beta_n q |u_j|^{\frac n{n-1}}} dx < \infty.
\]
Notice that for any $p < 1/(1-\|u_0\|_{1,\al}^n)^{1/(n-1)}$ we have
\begin{align*}
\lim_{j\to\infty} p\|\na u_j\|_n^{\frac n{n-1}} =p \lt(1+\al\| u_0\|_n^n\rt)^{\frac 1{n-1}}  & <\frac{\lt(1+\al\| u_0\|_n^n\rt)^{\frac 1{n-1}}}{\lt(1-\|\na u_0\|_n^n + \al \|u_0\|_n^n\rt)^{\frac1{n-1}}} \\
&= (1-\|\na v_0\|_n^n)^{-\frac1{n-1}}. 
\end{align*}
Thus we can choose a $q < (1-\|\na v_0\|_n^n)^{-\frac1{n-1}}$ and $j_0$ such that $p\|\na u_j\|_n^{\frac n{n-1}} \leq q$ for any $j\geq j_0$. Remark that 
\[
\int_\Om e^{\beta_n p |u_j|^{\frac n{n-1}}} dx =\int_{\Om} e^{\beta_n p \|\na u_j\|_n^{\frac n{n-1}} |v_j|^{\frac n{n-1}}} dx \leq \int_{\Om} e^{\beta_n q |v_j|^{\frac n{n-1}}} dx,
\]
for any $j\geq j_0$. The conclusion hence follows from the result of $\check{\rm C}$ern\'y, Cianchi and Hencl applied to the sequence $v_j$.
\end{proof}

\begin{proof}[Proof of Proposition \ref{subcritical}]
Let $\{u_j\}_j$ be a maximizing sequence for $C_\ep$. Since $\al < \lam_1(\Om)$ then
\[
1 =\|\na u_j\|_n^n -\alpha \|u_j\|_n^n \geq \lt(1 -\frac\al {\lam_1(\Om)}\rt) \|\na u_j\|_n^n.
\]
Hence $u_j$ is bounded in $W^{1,n}(\Om)$. By Sobolev embedding, we can assume that $u_j \rightharpoonup u_\ep$ weakly in $W^{1,n}(\Om)$, $u_j\to u_\ep$ in $L^p(\Om)$ for any $p < \infty$ and $u_j\to u_\ep$ a.e. in $\Om$. Evidently, $u_\ep\in \mathcal H$ and $\|u_\ep\|_{1,\alpha} \leq 1$. If $u_\ep\equiv 0$, by Lemma \ref{Lions} we can choose $1< q < \beta_n/\beta_\ep$ such that $e^{\beta_\ep |u_j|^{\frac n{n-1}}}$ is bounded in $L^q(\Om)$, hence
\[
C_\ep = \lim_{j\to\infty} \int_\Om e^{\beta_\ep |u_j|^{\frac{n}{n-1}}} dx = |\Om|,
\]
which is impossible. Thus, we have $u_\ep\not\equiv 0$. Using again Lemma \ref{Lions}, we can choose $q>1$ such that $e^{\beta_\ep |u_j|^{\frac n{n-1}}}$ is bounded in $L^q(\Om)$. Hence
\[
C_\ep = \lim_{j\to\infty} \int_\Om e^{\beta_\ep |u_j|^{\frac{n}{n-1}}} dx = \int_\Om e^{\beta_\ep |u_\ep|^{\frac{n}{n-1}}} dx.
\]
It remains to check that $\|u_\ep\|_{1,\al} =1$. Indeed, if otherwise then $\|u_\ep\|_{1,\al} <1$, denote $v_\ep = u_\ep/\|u_\ep\|_{1,\al}$ then $v_\ep \in \mathcal H$ and $\|v_\ep\|_{1,\al} =1$ and
\[
C_\ep = \int_\Om e^{\beta_\ep |u_\ep|^{\frac{n}{n-1}}} dx < \int_\Om e^{\beta_\ep |v_\ep|^{\frac{n}{n-1}}} dx  \leq C_\ep,
\]
which is impossible.

An easy and straightforward computation show that $u_\ep$ satisfies the Euler--Lagrange equation \eqref{eq:EL}. By standard elliptic regularity to \eqref{eq:EL}, we have $u_\ep \in C^1(\o{\Om})$. 

Obviously,
\[
\limsup_{\ep\to 0} C_\ep \leq \sup_{u\in \mathcal H, \|u\|_{1,\al}\leq 1} \int_\Om e^{\beta_n |u|^{\frac{n}{n-1}}} dx.
\]
For any $u\in \mathcal H$ with $\|u\|_{n,\al}\leq 1$, by using Fatou's lemma, we have
\[
\int_\Om e^{\beta_n |u|^{\frac{n}{n-1}}} dx \leq \liminf_{\ep\to 0} \int_\Om e^{\beta_\ep |u|^{\frac{n}{n-1}}} dx \leq \liminf_{\ep\to 0} C_\ep.
\]
Taking the supremum over all such functions $u$, we get
\[
\liminf_{\ep \to 0} C_\ep \geq \sup_{u\in \mathcal H, \|u\|_{1,\al}\leq 1} \int_\Om e^{\beta_n |u|^{\frac{n}{n-1}}} dx.
\]
Combining these two estimates together, we get \eqref{eq:limitCep}.

Using the inequality $e^t \leq 1 + t e^t$, we get
\[
C_\ep =\int_\Om e^{\beta_\ep |u_\ep|^{\frac{n}{n-1}}} dx \leq |\Om| +\beta_\ep \lambda_\ep.
\]
This together \eqref{eq:limitCep} implies
\[
\beta_n \liminf_{\ep \to 0} \lam_\ep \geq \sup_{u\in \mathcal H, \|u\|_{1,\al}\leq 1} \int_\Om e^{\beta_n |u|^{\frac{n}{n-1}}} dx -|\Om| >0,
\]
as \eqref{eq:liminf}.

Since the inequality $t^{1/(n-1)} e^{\be_\ep t^{n/(n-1)}} \leq e^{\be_\ep} + t^{n/(n-1)} e^{\beta_\ep t^{n/(n-1)}}$ holds for any $t \geq 0$, hence $|\mu_\ep| \leq e^{\be_\ep} + \lam_\ep/|\Om|$. This together \eqref{eq:liminf} proves the first inequality in \eqref{eq:upbound}. The second inequality in \eqref{eq:upbound} is trivial.
\end{proof}

\section{Asymptotic behavior of extremal functions for the subcritical inequalities}
Denote $c_\ep =\max_{\o{\Om}} |u_\ep|$. Without loss of generality, we can assume that $c_\ep = u_\ep(x_\ep)$, otherwise we consider $-u_\ep$ instead of $u_\ep$, and $x_\ep \to p \in \o{\Om}$. If $c_\ep$ is bounded, then by applying elliptic estimates to \eqref{eq:EL}, we get that $u_\ep \to u^*$ in $C^1(\o{\Om})$ for some function $u^*$. This convergence implies that Theorems \ref{Main1} and \ref{Main2} hold. In the rest of this section, we only consider the case $c_\ep \to \infty$. We do not distinguish the sequence and subsequence, the interest reader should understand it from the context.

Since $u_\ep$ is bounded in $W^{1,n}(\Om)$ then we can assume that $u_\ep \rightharpoonup u_0$ weakly in $W^{1,n}(\Om)$, $u_\ep \to u_0$ in $L^q(\Om)$ for any $q < \infty$ and $u_\ep \to u_0$ a.e. in $\Om$. If $u_0\not\equiv 0$, then there exists $r >1$ such that $e^{\beta_\ep |u_\ep|^{n/(n-1)}}$ is bounded in $L^r(\Om)$. Applying elliptic estimates to \eqref{eq:EL} we get $c_\ep$ is bounded which is impossible. Thus $u_0 \equiv 0$.

We next claim that $p\in \pa \Om$. Indeed, if $p \in \Om$, we can choose $r >0$ such that $B_{r}(p) \subset \Om$. Let $\phi$ be a cut-off function in $B_{r}(p)$, i.e., $\phi \in C_0^\infty(B_{r}(p))$, $0\leq \phi \leq 1$ and $\phi = 1$ in $B_{r/2}(p)$. Note that $\phi u_\ep \in W_0^{1,n}(\Om)$ and 
\begin{align*}
\int_\Om |\na(\phi u_\ep)|^n dx &= \int_\Om|\phi \na u_\ep + u_\ep \na \phi|^n dx\\
&\leq (1+ \de)\int_\Om |\na u_\ep|^n \phi^n dx + \lt(1-(1+\de)^{\frac{1}{1-n}}\rt)^{1-n} \int_{\Om} |\na \phi|^n |u_\ep|^n dx\\
&\leq (1+ \de) + \lt[\al(1+\de) + C^n\lt(1-(1+\de)^{\frac{1}{1-n}}\rt)^{1-n}\rt] \|u_\ep\|_n^n,
\end{align*}
for any $\de >0$, where $C = \sup |\na \phi|$. Fix $\de < 1/4$, then for $\ep >0$ small enough, we get $\|\na(\phi u_\ep)\|_n^n \leq 1+2\de < 3/2$. Applying classical Moser--Trudinger inequality \eqref{eq:MT}, there exists $q >1$ such that $e^{\beta_\ep |\phi u_\ep|^{\frac n{n-1}}}$ is bounded in $L^q(\Om)$. In particular, $e^{\beta_\ep |u_\ep|^{\frac n{n-1}}}$ is bounded in $L^q(B_{r/2}(p)$. Using elliptic estimates to \eqref{eq:EL} in $B_{r/2}(p)$ we get that $u_\ep$ is bounded in $C^1(\o{B_{r/4}(p)})$. Hence $c_\ep$ is bounded which is impossible.

We next prove that 
\begin{equation}\label{eq:Dirac}
|\na u_\ep|^n dx \rightharpoonup  \de_p \qquad\text{\rm in measure sense}.
\end{equation}
Indeed, we have $\|\na u_\ep\|_n^n \to 1$ as $\ep \to 0$. Hence, if \eqref{eq:Dirac} does not hold, then there exists $\mu < 1$ and $r>0$ small such that
\[
\lim_{\ep\to 0} \int_{\Om \cap B_r(p)} |\na u_\ep|^n dx \leq \mu.
\]
Consider again cut-off function $\phi$ as above, and define $\phi_\ep =\phi u_\ep -\frac1{|\Om|} \int_\Om \phi u_\ep dx$. Thus $\phi_\ep \in \mathcal H$, and 
\begin{align*}
\int_\Om |\na \phi_\ep|^n dx &=\int_{\Om} |\phi\na u_\ep + u_\ep \na \phi|^n dx\\
&\leq (1+\de) \int_\Om |\na u_\ep|^n dx + C^n\lt(1 -(1+\de)^{\frac1{1-n}}\rt)^{1-n} \int_\Om |u_\ep|^n dx,
\end{align*}
for any $\de >0$. Fix a $\de >0$ such that $\de < (1-\mu)/(2\mu)$, we have
\[
\limsup_{\ep\to 0} \int_\Om |\na \phi_\ep|^n dx \leq (1+\de) \mu < \frac{1+\mu} 2.
\]
Thus for $\ep >0$ small enough, we get $\|\na \phi_\ep\|_n^n < (1+ \mu)/2 < 1$. By Cianchi's inequality \eqref{eq:Cianchi}, $e^{\beta_\ep |\phi_\ep|^{\frac n{n-1}}}$ is bounded in $L^q(\Om)$ for some $q >1$. We again have
\[
|\phi u_\ep|^{\frac n{n-1}} \leq (1+t)|\phi_\ep|^{\frac{n}{n-1}} + \lt(1 -(1+t)^{1-n}\rt)^{\frac1{1-n}} \lt|\frac1{|\Om|} \int_\Om \phi u_\ep dx\rt|^{\frac n{n-1}},
\]
for any $t>0$. The second term on the right hand side tends to zero as $\ep \to 0$. Hence by choose $t >0$ small enough, we have that $e^{\beta_\ep |\phi u_\ep|^{\frac n{n-1}}}$ is bounded in $L^{q'}(\Om)$ for some $q'>1$. In particular, $e^{\beta_\ep |u_\ep|^{\frac n{n-1}}}$ is bounded in $L^{q'}(\Om \cap B_{r/2}(p))$. Note that $\pa_\nu u_\ep =0$ on $\pa \Om$, by applying elliptic estimates to \eqref{eq:EL} in $\Om \cap B_{r/2}(p)$, we get that $u_\ep$ is bounded near $p$ which is impossible.

Denote $r_\ep^n = \lam_\ep c_\ep^{-\frac n{n-1}} e^{-\beta_\ep c_\ep^{\frac n{n-1}}}$. We then have $\lim_{\ep \to 0} r_\ep =0$. Indeed, for any $0< \gamma < \beta_n$, we have $\beta_\ep -\gamma >0$ for $\ep >0$ small enough. Hence
\begin{equation}\label{eq:convergerep}
r_\ep^n c_\ep^{\frac{n}{n-1}} e^{\gamma c_\ep^{\frac n{n-1}}} \leq \int_\Om e^{\gamma |u_\ep|^{\frac n{n-1}}} |u_\ep|^{\frac n{n-1}} dx \to 0, 
\end{equation}
here we use H\"older inequality, \eqref{eq:Cianchi} and the fact $u_\ep \to 0$ in $L^q(\Om)$ for any $q < \infty$.

We continue studying the asymptotic behavior of $u_\ep$ near $p$. Following the argument in \cite{Yang07} we take $(V,\phi)$ a normal coordinate system around $p$ such that $\phi(p) = 0$, $\phi(\pa \Om \cap V) = \{y\in \R^n: y_1 =0\} \cap B_1(0)$ and $\phi(\Om \cap V) = \{y\in \R^n: y_1 >0\} \cap B_1(0)$. In this coordinate, the original metric $g =dx_1^2 +\cdots dx_n^2$ has the form $g =\sum_{i,j =1}^ng_{ij} dy_idy_j$ with 
\[
g_{ij} = g_{ij}(y) = \sum_{k=1}^n \frac{\pa {\phi^{-1}}^k}{\pa y_i}\frac{\pa {\phi^{-1}}^k}{\pa y_j},\quad g_{ij}(0) = \de_{ij}, \quad \frac{\pa g_{ij}}{\pa y_l}(0) =0,
\]
for any $i,j,l$. We also use $g$ to denote matrix $(g_{ij})_{n\times n}$ and use $(g^{ij})_{n\times n}$ to denote the inverse of $g$. In this coordinate system, we have the following relation: for a function $f$ on $V$, denote $h =f\circ \phi^{-1}$ the function on $B_1(0)$, then $|\na f(x)| = |\na_g h (\phi(x))|_g$ and $\Delta_n f(x) = \Delta_{g,n} h(\phi(x))$ where
\begin{equation}\label{eq:nlap}
\Delta_{g,n} h = \frac1{\sqrt{\text{\rm det}(g)}}\sum_{i,j=1}^n\frac{\pa}{\pa y_i}\lt[g^{ij}\sqrt{\text{\rm det}(g)} \lt(\sqrt{\sum_{k,l =1}^n g^{kl} \frac{\pa h}{\pa y_k} \frac{\pa h}{\pa y_l}}\rt)^{n-2} \frac{\pa h}{\pa y_j}\rt],
\end{equation}
is $n-$Laplace with respect to $g$. Let us define the function $\tilde u_\ep$ on $B_1(0)$ by
\[
\tilde u_\ep(y) = 
\begin{cases}
u_\ep \circ \phi^{-1}(y_1,y')&\mbox{if $y_1 \geq 0$,}\\
u_\ep \circ \phi^{-1}(-y_1,y') &\mbox{if $y_1 < 0$,}
\end{cases}
\]
here we write $y\in \R^n$ by $(y_1,y')$. From \eqref{eq:nlap}, we see that $\tilde u_\ep$ satisfies
\[
-\De_{g,n}\tilde u_\ep = \frac1{\lam_\ep} e^{\beta_\ep |\tilde u_\ep|^{\frac n{n-1}}} |\tilde u_\ep|^{\frac{2-n}{n-1}}\tilde u_\ep + \al |\tilde u_\ep|^{n-2} \tilde u_\ep -\frac{\mu_\ep+\al \lam_\ep \nu_\ep}{\lam_\ep},
\] 
on $B_1(0)$.

Denote $y_\ep = \phi(x_\ep)$ and $\Om_\ep =\{y\in \R^n: y_\ep + r_\ep y\in B_1(0)\}$. we define two sequences of functions on $\Om_\ep$ by
\[
\psi_\ep(y) = \frac1{c_\ep} \tilde u_\ep(y_\ep + r_\ep y),\qquad \vphi_\ep =c_\ep^{\frac1{n-1}} (\tilde u_\ep(y_\ep +r_\ep y) -c_\ep).
\]
Then we have
\begin{equation}\label{eq:psiep}
-\Delta_{g,n}\psi_\ep =c_\ep^{-n} \psi_\ep |\psi_\ep|^{n-2} e^{\beta_\ep(|\tilde u_\ep|^{\frac n{n-1}}-c_\ep^{\frac{n}{n-1}})}+\alpha r_\ep^{n} |\psi_\ep|^{n-2} \psi_\ep -\frac{r_\ep^n}{c_\ep^{n-1}}\frac{\mu_\ep+\al \lam_\ep \nu_\ep}{\lam_\ep}.
\end{equation}
and
\begin{equation}\label{eq:vphiep}
-\De_{g,n} \vphi_\ep=\psi_\ep |\psi_\ep|^{n-2} e^{\beta_\ep(|\tilde u_\ep|^{\frac n{n-1}}-c_\ep^{\frac{n}{n-1}})}+\alpha c_\ep^nr_\ep^{n} |\psi_\ep|^{n-2} \psi_\ep -c_\ep r_\ep^n\frac{\mu_\ep+\al \lam_\ep \nu_\ep}{\lam_\ep}
\end{equation}
on $\Om_\ep$.
\begin{lemma}\label{convergepsiep}
It holds $\psi_\ep \to 1$ in $C^1_{\rm loc}(\R^n)$.
\end{lemma}
\begin{proof}
It follows from \eqref{eq:upbound} and \eqref{eq:psiep} that
\[
|\Delta_{g,n} \psi_\ep| \leq c_\ep^{-n} + \al r_\ep^n + c \frac{r_\ep^n}{c_\ep^{n-1}} \to 0,
\]
as $\ep \to 0$ and $\psi_\ep \leq \psi_\ep (0) =1$. Applying elliptic estimates and the Liouville theorem for $n-$harmonic functions, we get the conclusion.
\end{proof}

\begin{lemma}\label{convergevphiep}
It holds $\vphi_\ep \to \vphi$ in $C^1_{\rm loc}(\R^n)$ with
\begin{equation}\label{eq:vphifunct}
\vphi(x) =-\frac{n-1}{\beta_n} \ln \lt(1+ \lt(\frac{\om_{n-1}}{2n}\rt)^{\frac1{n-1}} |x|^{\frac n{n-1}}\rt).
\end{equation}
\end{lemma}
\begin{proof}
Fix a $R >0$. Since $y_\ep, r_\ep \to 0$, hence $y_\ep + r_\ep B_R(0) \subset B_1(0)$ for $\ep$ small enough. Applying the Hacnack inequality for an $n-$Laplace equation \cite{Serrin64} and \eqref{eq:upbound}, \eqref{eq:convergerep} and Lemma \ref{convergepsiep} to equation \eqref{eq:vphiep}, we get that $\vphi_\ep$ is bounded in $L^\infty(B_R(0))$. Then by elliptic estimates \cite{Tolksdorf}, we obtain that $\vphi_\ep$ is bounded in $C^{1,\gamma}(B_{R/2}(0))$ for some $0 < \gamma < 1$, whence $\vphi_\ep \to \vphi$ in $C^1(B_{R/4}(0))$. Since $R >0$ is arbitrary, then $\vphi_\ep \to \vphi$ in $C^1_{\rm loc}(\R^n)$.

It remains to find the form of $\vphi$. By Lemma \ref{convergepsiep}, we have
\begin{align}\label{eq:1}
|\tilde u_\ep(y_\ep +r_\ep y)|^{\frac n{n-1}}-c_\ep^{\frac{n}{n-1}}&=c_\ep^{\frac n{n-1}} \lt(|\psi_\ep(y)|^{\frac{n}{n-1}} -1\rt)\notag\\
&=c_\ep^{\frac n{n-1}} \lt( (1+ (\psi_\ep -1))^{\frac{n}{n-1}} -1\rt)\notag\\
&= c_\ep^{\frac n{n-1}} \lt(\frac n{n-1}(\psi_\ep -1) + O((\psi_\ep -1)^2)\rt)\notag\\
&= \frac n{n-1} \vphi_\ep + O(|\psi_\ep -1|),
\end{align}
uniformly in $B_R(0)$. Notice that $g(y_\ep + r_\ep y) \to (\de_{ij})_{n\times n}$ uniformly in $B_R(0)$ when $\ep \to 0$. This together \eqref{eq:vphiep}, \eqref{eq:upbound}, \eqref{eq:convergerep} and \eqref{eq:1} shows that $\vphi$ satisfies
\begin{equation}\label{eq:vphiequation}
\begin{cases}
-\De_n \vphi = e^{\frac n{n-1}\beta_n \vphi}&\mbox{in $\R^n$,}\\
\vphi(x) \leq \vphi(0) =0&\mbox{$\forall\, x\in \R^n$.}
\end{cases}
\end{equation}
Moreover, for any $R >0$, by \eqref{eq:1} we have
\begin{align*}
\int_{B_R(0)} e^{\frac{n}{n-1}\beta_n \vphi} dy &= \lim_{\ep\to 0} \int_{B_R(0)} e^{\beta_\ep(|\tilde u_\ep(y_\ep + r_\ep y)|^{\frac n{n-1}} -c_\ep^{\frac n{n-1}})} dy\\
&=\lim_{\ep \to 0} \frac{c_\ep^{\frac n{n-1}} \int_{B_{Rr_\ep}(y_\ep)} e^{\beta_\ep |\tilde u_\ep(y)|^{\frac n{n-1}}} dy}{\int_\Om |u_\ep|^{\frac n{n-1}} e^{\beta_\ep |u_\ep|^{\frac n{n-1}}}dx}\\
&\leq \lim_{\ep \to 0} \frac{c_\ep^{\frac n{n-1}} \int_{B_{Rr_\ep}(y_\ep)} e^{\beta_\ep |\tilde u_\ep(y)|^{\frac n{n-1}}} dy}{\int_{B_{Rr_\ep}(y_\ep)\cap \{y: y_1>0\}} |\tilde u_\ep|^{\frac n{n-1}} e^{\beta_\ep |\tilde u_\ep|^{\frac n{n-1}}} \sqrt{\text{\rm det}(g)} dy}\\
&=\lim_{\ep \to 0} (1+o_{\ep,R}(1))\frac{\int_{B_{Rr_\ep}(y_\ep)} e^{\beta_\ep |\tilde u_\ep(y)|^{\frac n{n-1}}} dy}{\int_{B_{Rr_\ep}(y_\ep)\cap \{y: y_1>0\}} e^{\beta_\ep |\tilde u_\ep|^{\frac n{n-1}}} dy},\\
&=\lim_{\ep \to 0}\frac{\int_{B_{R}(0)} e^{\beta_\ep (|\tilde u_\ep(y_\ep +r_\ep y)|^{\frac n{n-1}}-c_\ep^{\frac{n}{n-1}})} dy}{\int_{B_{R}(0)\cap \{y: y_1>-\frac{{y_\ep}_1}{r_\ep}\}} e^{\beta_\ep (|\tilde u_\ep(y_\ep +r_\ep y)|^{\frac n{n-1}}-c_\ep^{\frac n{n-1}})} dy}
\end{align*}
here $o_{\ep,R}(1) \to 0$ as $\ep \to 0$ and $R$ is fixed and ${y_\ep}_1$ is the first coordinate of $y_\ep$. Suppose that ${y_\ep}_1/r_\ep \to a \geq 0$ as $\ep \to 0$, then
\[
\int_{B_R(0)} e^{\frac{n}{n-1}\beta_n \vphi} dy \leq \frac{\int_{B_{R}(0)} e^{\frac n{n-1} \beta_n \vphi} dy}{\int_{B_{R}(0)\cap \{y: y_1>-a\}} e^{\frac n{n-1} \beta_n \vphi} dy} \leq 2.
\]
Letting $R\to \infty$, we get $\int_{\R^n} e^{\frac{n}{n-1}\beta_n \vphi} dy \leq 2$. Using the argument at the end of the proof of Lemma $3.6$ in \cite{Yang07} or applying a recent classification result of Esposito \cite{Esposito}, we get the form of $\vphi$ as \eqref{eq:vphifunct}. 
\end{proof}
Notice that $\int_{\R^n} e^{\frac n{n-1}\beta_n \vphi} dy =2$ and hence the argument in the proof of Lemma \ref{convergevphiep} above implies that ${y_\ep}_1/r_\ep \to 0$ as $\ep \to 0$.

For $c >1$, denote $u_{\ep,c} =\min\{u_\ep, c_\ep/c\}$ we have the following
\begin{lemma}\label{truncation}
It holds $\lim_{\ep\to 0} \int_\Om |\na u_{\ep,c}|^n dx = 1/c$ for any $c >1$.
\end{lemma}
\begin{proof}
The proof is completely analogous to \cite{Li2005}, so we omit it.
\end{proof}
\begin{lemma}\label{chantrensup}
It holds
\[
\sup_{u\in \mathcal H, \|u\|_{1,\al}\leq 1} \int_\Om e^{\beta_n |u|^{\frac{n}{n-1}}} dx \leq |\Om| + \limsup_{\ep \to 0} \frac{\lam_\ep}{c_\ep^{\frac n{n-1}}}.
\]
\end{lemma}
\begin{proof}
Fix $c >1$ and define $u_{\ep,c}$ as above. Lemma \ref{truncation} implies
\[
\lim_{\ep\to 0} \|\na u_{\ep,c}\|_n^n = \frac 1c < 1.
\]
By Cianchi's inequality \eqref{eq:Cianchi}, $e^{\beta_\ep |u_{\ep,c}|^{\frac n{n-1}}}$ is bounded in $L^q(\Om)$ for some $q >1$ as $\ep$ small enough. Since $u_{\ep,c} \to 0$ a.e. in $\Om$, then
\[
\lim_{\ep\to 0} \int_{\Om} e^{\beta_\ep |u_{\ep,c}|^{\frac n{n-1}}} dx = |\Om|.
\]
We have
\begin{align*}
\int_{\Om} e^{\beta_\ep |u_{\ep}|^{\frac n{n-1}}} dx &= \int_{\{u_\ep \leq c_\ep/c\}}e^{\beta_\ep |u_{\ep}|^{\frac n{n-1}}} dx + \int_{\{u_\ep > c_\ep/c\}}e^{\beta_\ep |u_{\ep}|^{\frac n{n-1}}} dx\\
&\leq \int_{\Om} e^{\beta_\ep |u_{\ep,c}|^{\frac n{n-1}}} dx + c^{\frac n{n-1}} \frac{\lam_\ep}{c_\ep^{\frac{n}{n-1}}}.
\end{align*}
Let $\ep \to 0$, $c\to 1$ and using \eqref{eq:limitCep} we obtain the desired result.
\end{proof}
As an easy consequence of Lemma \ref{chantrensup} we have $\lim_{\ep\to 0} c_\ep/\lam_\ep = 0$. Indeed, if this is not the case, then we obtain from Lemma \ref{chantrensup} that
\[
\sup_{u\in \mathcal H, \|u\|_{1,\al}\leq 1} \int_\Om e^{\beta_n |u|^{\frac{n}{n-1}}} dx \leq |\Om|
\]
which is impossible. Also, we have $c_\ep^{\frac n{n-1}}/\lam_\ep$ is bounded.

We continue by studying the asymptotic behavior of $u_\ep$ away from the blow up point $p$. We have the following result
\begin{lemma}\label{boundedinSobolev}
$c_\ep^{\frac1{n-1}} u_\ep$ is bounded in $H^{1,q}(\Om)$ for any $1< q <n$, and $c_\ep^{\frac1{n-1}} u_\ep \rightharpoonup G$ weakly in $W^{1,q}(\Om)$ for any $1< q <n$, where $G$ is a Green function satisfying
\begin{equation}\label{eq:Green}
\begin{cases}
-\De_n G = \de_p + \alpha\lt(|G|^{n-2}G - \frac1{|\Om|}\int_\Om |G|^{n-2}Gdx\rt) -\frac1{|\Om|}&\mbox{in $\o{\Om}$,}\\
\pa_\nu G = 0&\mbox{on $\pa \Om \setminus\{p\}$,}\\
\int_\Om G dx =0.
\end{cases}
\end{equation}
Furthermore, $c_\ep^{\frac 1{n-1}} u_\ep\to G$ in $C^1(\o{\Om'})$ for any $\Om'\subset\subset \o{\Om}\setminus\{p\}$, and $G$ has form
\begin{equation}\label{eq:Gform}
G(x) = -\frac n{\beta_n}\ln |x-p| + A_p + \beta(x),
\end{equation}
where $A_p$ is constant, and $\beta \in C^0(\o{\Om}) \cap C^1(\o{\Om}\setminus\{p\})$ and $\beta(x) = O(|x-p|)$ as $x\to p$.
\end{lemma}
\begin{proof}
We first claim that
\begin{equation}\label{eq:claim}
\frac{c_\ep}{\lam_\ep} |u_\ep|^{\frac{2-n}{n-1}} u_\ep e^{\beta_\ep |u_\ep|^{\frac n{n-1}}} \rightharpoonup  \de_p,
\end{equation}
weakly. Indeed, fix a $c >1$ and $R >0$, we divide $\Om$ into three parts as follows
\[
\Om_1 =\{u_\ep > c_\ep/c\}\setminus \phi^{-1}(B_{Rr_\ep}(y_\ep)),\, \Om_2 = \{u_\ep \leq c_\ep/c\},\, \Om_3 = \Om \cap \phi^{-1}(B_{Rr_\ep}(y_\ep)),
\] 
where $(V,\phi)$ denotes the coordinate system around $p$ above. By Lemma \ref{convergepsiep}, we get $\phi^{-1}(B_{Rr_\ep}(y_\ep))\cap \Om \subset \{u_\ep > c_\ep/c\}$ for $\ep$ small enough. For any $\psi \in C^1(\o{\Om})$ we have
\begin{align*}
\Bigg|\int_{\Om_1} \frac{c_\ep}{\lam_\ep} |u_\ep|^{\frac{2-n}{n-1}}& u_\ep e^{\beta_\ep |u_\ep|^{\frac n{n-1}}} \psi dx\Bigg| \leq \sup_{\Om}|\psi| \int_{\Om_1} \frac{c_\ep}{\lam_\ep} u_\ep^{\frac{1}{n-1}} e^{\beta_\ep |u_\ep|^{\frac n{n-1}}} dx \\
&=\sup_{\Om} |\psi| \lt(\int_{\{u_\ep > \frac {c_\ep}c\}}\frac{c_\ep}{\lam_\ep} u_\ep^{\frac{1}{n-1}} e^{\beta_\ep |u_\ep|^{\frac n{n-1}}} dx -\int_{\Om_3} \frac{c_\ep}{\lam_\ep} u_\ep^{\frac{1}{n-1}} e^{\beta_\ep |u_\ep|^{\frac n{n-1}}} dx\rt)\\
&\leq \sup_{\Om} |\psi| \lt(c -\int_{B_{Rr_\ep}(y_\ep)\cap\{y: y_1 >0\}}\frac{c_\ep}{\lam_\ep} \tilde u_\ep^{\frac{1}{n-1}} e^{\beta_\ep |\tilde u_\ep|^{\frac n{n-1}}}\sqrt{\text{\rm det}(g)} dy\rt)\\
&\leq\sup_{\Om} |\psi| \lt(c -\int_{B_{R}(0)\cap\{y: y_1 >0\}}\psi_\ep^{\frac 1{n-1}} e^{\beta_\ep (|\tilde u_\ep|^{\frac n{n-1}} -c_\ep^{\frac{n}{n-1}})}\sqrt{\text{\rm det}(g)(y_\ep +r_\ep y)} dy\rt)\\
&=\sup_{\Om} |\psi| \lt(c -\int_{B_{R}(0)\cap\{y: y_1 \geq 0\}}e^{\frac n{n-1}\beta_n \vphi} dy + o_{\ep,R}(1)\rt),
\end{align*}
here we use Lemmas \ref{convergepsiep} and \ref{convergevphiep} and the fact $g(y_\ep + r_\ep y) \to (\de_{ij})_{n\times n}$ uniformly in $B_R(0)$. Thus
\begin{equation}\label{eq:Om1}
\int_{\Om_1} \frac{c_\ep}{\lam_\ep} |u_\ep|^{\frac{2-n}{n-1}} u_\ep e^{\beta_\ep |u_\ep|^{\frac n{n-1}}} \psi dx = O(c-1)+ o_\ep(1) + o_{R}(1).
\end{equation}
On $\Om_2$ we have
\[
\lt|\int_{\Om_2} \frac{c_\ep}{\lam_\ep} |u_\ep|^{\frac{2-n}{n-1}} u_\ep e^{\beta_\ep |u_\ep|^{\frac n{n-1}}} \psi dx\rt| \leq \sup_{\Om}|\psi| \frac{c_\ep}{\lam_\ep} \int_{\Om} |u_{\ep,c}|^{\frac1{n-1}} e^{\beta_\ep |u_{\ep,c}|^{\frac n{n-1}}} dx
\]
The integral is bounded uniformly in $\ep$ by Lemma \ref{truncation}. This together the remark after Lemma \ref{chantrensup} implies
\begin{equation}\label{eq:Om2}
\int_{\Om_2} \frac{c_\ep}{\lam_\ep} |u_\ep|^{\frac{2-n}{n-1}} u_\ep e^{\beta_\ep |u_\ep|^{\frac n{n-1}}} \psi dx = o_{\ep,c}(1).
\end{equation}
On $\Om_3$ we have
\begin{align*}
\int_{\Om_3} \frac{c_\ep}{\lam_\ep} |u_\ep|^{\frac{2-n}{n-1}} u_\ep e^{\beta_\ep |u_\ep|^{\frac n{n-1}}} \psi dx&= \int_{B_{Rr_\ep}(y_\ep)\cap\{y:y_1>0\}}\frac{c_\ep}{\lam_\ep} \tilde u_\ep^{\frac1{n-1}}  e^{\beta_\ep \tilde u_\ep^{\frac n{n-1}}} \psi\circ \phi^{-1} \sqrt{\text{\rm det}(g)} dy\\
&=(\psi(p) +o_{\ep,R}(1))\int_{B_R(0)\cap \{y: y_1 > \frac{{y_\ep}_1}{r_\ep}\}} \psi_\ep^{\frac1{n-1}}e^{\beta_\ep (|\tilde u_\ep|^{\frac n{n-1}} -c_\ep^{\frac{n}{n-1}})} dy\\
&=(\psi(p) +o_{\ep,R}(1))\lt(\int_{B_R(0)\cap \{y: y_1 \geq 0\}} e^{\frac n{n-1} \beta_n \vphi} dy + o_{\ep,R}(1)\rt),
\end{align*}
here we use Lemmas \ref{convergepsiep} and \ref{convergevphiep} and the facts ${y_\ep}_1/r_\ep \to 0$, and $g(y_\ep + r_\ep y) \to (\de_{ij})_{n\times n}$ uniformly in $B_R(0)$. Thus
\begin{equation}\label{eq:Om3}
\int_{\Om_3} \frac{c_\ep}{\lam_\ep} |u_\ep|^{\frac{2-n}{n-1}} u_\ep e^{\beta_\ep |u_\ep|^{\frac n{n-1}}} \psi dx = \psi(p) + o_{\ep,R}(1).
\end{equation}
Combining \eqref{eq:Om1}, \eqref{eq:Om2} and \eqref{eq:Om3} proves our claim \eqref{eq:claim}. 

Taking $\psi \equiv 1$, we obtain
\begin{equation}\label{eq:gioihanhang}
\lim_{\ep\to 0} \frac{c_\ep \mu_\ep}{\lam_\ep} = \frac1{|\Om|}.
\end{equation}

Fix a $c >1$, we have
\begin{align*}
\int_\Om \frac{c_\ep}{\lam_\ep} |u_\ep|^{\frac{1}{n-1}}  e^{\beta_\ep |u_\ep|^{\frac n{n-1}}} dx &= \int_{\{u_\ep \leq c_\ep/c\}} \frac{c_\ep}{\lam_\ep} |u_\ep|^{\frac{1}{n-1}}  e^{\beta_\ep |u_\ep|^{\frac n{n-1}}} dx + \int_{\{u_\ep > c_\ep/c\}} \frac{c_\ep}{\lam_\ep} |u_\ep|^{\frac{1}{n-1}}  e^{\beta_\ep |u_\ep|^{\frac n{n-1}}} dx\\
&\leq \frac{1}{c^{\frac1{n-1}}} \frac{c_\ep^{\frac n{n-1}}}{\lam_\ep} \int_\Om e^{\beta_\ep |u_{\ep,c}|^{\frac n{n-1}}} + c.
\end{align*} 
The remark after Lemma \ref{chantrensup} says that $\frac{c_\ep^{\frac n{n-1}}}{\lam_\ep}$ is bounded. This together with Lemma \ref{truncation} and Cianchi's inequality \eqref{eq:Cianchi} implies 
\begin{equation}\label{eq:L1bounded}
\limsup_{\ep \to 0} \int_\Om \frac{c_\ep}{\lam_\ep} |u_\ep|^{\frac{1}{n-1}}  e^{\beta_\ep |u_\ep|^{\frac n{n-1}}} dx \leq c + \frac{|\Om|}{c^{\frac1{n-1}}} \limsup_{\ep\to 0} \frac{c_\ep^{\frac n{n-1}}}{\lam_\ep} < \infty,
\end{equation}

Denote $w_\ep =c_\ep^{\frac1{n-1}} u_\ep$, from \eqref{eq:EL}, we have
\begin{equation}\label{eq:wep}
\begin{cases}
-\De_n w_\ep -\al\lt[|w_\ep|^{n-2} w_\ep -\frac{\int_{\Om} |w_\ep|^{n-2} w_\ep dx}{|\Om|}\rt] =\frac{c_\ep}{\lam_\ep} |u_\ep|^{\frac{2-n}{n-1}} u_\ep e^{\beta_\ep |u_\ep|^{\frac n{n-1}}} -\frac{c_\ep \mu_\ep}{\lam_\ep}&\mbox{in $\Om$,}\\
\pa_\nu w_\ep =0 &\mbox{on $\pa \Om$.}
\end{cases}
\end{equation}
We would like to show that $w_\ep$ is bounded in $H^{1,q}(\Om)$ for any $1< q < n$. Remark that 
\[
f_\ep:=\frac{c_\ep}{\lam_\ep} |u_\ep|^{\frac{2-n}{n-1}} u_\ep e^{\beta_\ep |u_\ep|^{\frac n{n-1}}} -\frac{c_\ep \mu_\ep}{\lam_\ep}
\]
is bounded in $L^1(\Om)$ by \eqref{eq:gioihanhang} and \eqref{eq:L1bounded}. We recall the following phenomena which was first discovered by Brezis and Merle \cite{BL}, developed by Struwe \cite{Struwe} and generalized on Riemannian manifolds by Li \cite{Li2005}: \emph{If $u\in W^{1,n}(\Om)$ be a weak solution of $-\De_n u = f$, $\int_\Om u dx =0$ then for any $1< q < n$ there exists $C(q)$ such that $\|\na w\|_q \leq C(q) \|f\|_1^{\frac1{n-1}}.$}

We will apply this observation to \eqref{eq:wep}. We argue as in \cite{Yang07corrige}. We first show that $w_\ep$ is bounded in $L^{n-1}(\Om)$. Indeed, if this is not the case, then $\|w_\ep\|_{n-1} \to \infty$. Define $v_\ep = w_\ep/\|w_\ep\|_{n-1}$, then $v_\ep$ satisfies
\[
\begin{cases}
-\De_n v_\ep = \alpha\lt(|v_\ep|^{n-2} v_\ep -\frac1{|\Om|} \int_\Om |v_\ep|^{n-2} v_\ep dx\rt) + \frac{f_\ep}{\|w_\ep\|_{n-1}^{n-1}} =: g_\ep&\mbox{in $\Om$,}\\
\pa_\nu v_\ep = 0&\mbox{on $\pa \Om$.}
\end{cases}
\]
Since $\|v_\ep\|_{n-1} =1$ and $f_\ep$ is bounded in $L^1(\Om)$ then so is $g_\ep$. Obviously $\int_\Om v_\ep dx =0$ by \eqref{eq:EL}. The observation above shows that $\|\na v_\ep\|_q$ is bounded for any $1 < q < n$. The mean value of $v_\ep$ is zero, by Poincar\'e inequality, $v_\ep$ is bounded in $W^{1,q}(\Om)$ for any $1< q< n$. Hence $v_\ep \rightharpoonup v$ weakly in $W^{1,q}(\Om)$ for any $1< q < n$, and $v_\ep \to v$ in $L^{n-1}(\Om)$. Therefore $\|v\|_{n-1} =1$ and $\int_\Om v dx =0$. It is easy to show that $v$ is weak solution of 
\[
\begin{cases}
-\De_n v = \al\lt(|v|^{n-2} v -\frac1{|\Om|} \int_\Om |v|^{n-2} v dx\rt)&\mbox{in $\Om$,}\\
\pa_\nu v =0&\mbox{on $\pa \Om$.}
\end{cases}
\]
Applying elliptic estimate to this equation, we get $v \in C^1(\o{\Om})$. Taking $v$ as a test function, we get $\|\na v\|_n^n = \al\|v\|_n^n$ (recall that $\int_\Om v dx =0$). Since $\al < \lam_1(\Om)$ then $v$ must be zero function which is impossible. Thus $w_\ep$ is bounded in $L^{n-1}(\Om)$. Consequently, $-\De_n w_\ep$ is bounded in $L^1(\Om)$ which then implies the boundedness of $c_\ep^{\frac1{n-1}} u_\ep$ in $W^{1,q}(\Om)$ for any $1<q<n$ by the observation  above of Brezis, Merle, Struwe and Li.

The rest of proof is similar with the one of Theorem $4.7$ in \cite{Li2005}.
\end{proof}

\section{Capacity estimates}
In this section, we use the capacity technique to give an upper bound of
\[
\sup_{u\in \mathcal H, \|u\|_{1,\alpha} \leq 1} \int_\Om e^{\beta_n |u|^{\frac n{n-1}}} dx,
\]
under the condition that $c_\ep\to \infty$, i.e., the blow-up occurs. We mention here that the technique of using capacity estimate applied to this kind of problems was discovered by Li \cite{Li2001} in dealing with Moser--Trudinger inequality. Our main result of this section is as follows
\begin{proposition}
Under the assumption that $c_\ep \to \infty$ as $\ep \to 0$, it holds
\begin{equation}\label{eq:suptren}
\sup_{u\in \mathcal H, \|u\|_{1,\alpha} \leq 1} \int_\Om e^{\beta_n |u|^{\frac n{n-1}}} dx \leq |\Om| + \frac{\om_{n-1}}{2n} e^{\beta_n A_p + 1 + \frac12+\cdots+\frac1{n-1}}.
\end{equation}
\end{proposition}
\begin{proof}
We follow the argument in \cite{Li2005,Yang07}. Consider a coordinate system $(V,\phi)$ around $p$ as in Section \S3. We write a vector $y\in \R^n$ by $(y_1,y')$. Denote $\o{x}_\ep = \phi^{-1}(0,y_\ep')\in \pa \Om$. Let $G_\ep$ be a distributional solution of
\begin{equation}\label{eq:Gep}
\begin{cases}
-\De_n G_\ep(x) =\de_{\o{x}_\ep} &\mbox{in $\o{\Om} \cap B_\de(\o{x}_\ep)$,}\\
G_\ep = -\frac{n}{\beta_n} \ln \de &\mbox{on $\Om \cap \pa B_\de(\o{x}_\ep)$,}\\
\pa_\nu G_\ep =0&\mbox{on $\pa \Om \cap B_\de(\o{x}_\ep)$}.
\end{cases}
\end{equation}
It was shown by Kichennassamy and Veron \cite{Kichenassamy} and by Li \cite{Li2005}, using a reflection argument, that $G_\ep$ exists and has the form
\[
G_\ep(x) = -\frac{n}{\beta_n} \ln |x-\o x_\ep| + v_\ep(x),
\]
where $v_\ep = O(\de)$ uniformly with respect to $\ep$.

For $c_1 \leq c_2$ we define a space of functions $\Lam_\ep(c_1,c_2,a,b)$ by
\begin{align*}
\Lam_\ep(c_1,c_2,a,b)& =\Bigl\{u\in W^{1,n}(\{x\in \Om: c_1\leq G_\ep(x) \leq c_2\}): u\bigl|_{G_\ep =c_1} =a,\\
&\qquad\qquad\qquad\qquad\qquad\qquad\qquad \, u\bigl|_{G_\ep =c_2} =b, \,\pa_{\nu} u\bigl|_{\pa \Om} =0\Bigl\}.
\end{align*}
It was shown in \cite{Yang07} that $\inf_{\Lam_\ep(c_1,c_2,a,b)} \int_{c_1\leq G_\ep \leq c_2}|\na u|^n dx$ is attained by a function $\Psi$ having the form
\begin{equation}\label{eq:Psiform}
\Psi = \frac{b(G_\ep -c_1) -a(G_\ep -c_2)}{c_2 -c_1}
\end{equation}
and satifying
\begin{equation}\label{eq:energy}
\int_{c_1 \leq G_\ep \leq c_2} |\na \Psi|^n dx = \frac{|b-a|^n}{(c_2 -c_1)^{n-1}}.
\end{equation}
Choose $y_\ep \in \Om \cap B_\de(\o{x}_\ep)$ such that $|y_\ep -\o{x}_\ep| = Rr_\ep$. Set
\[
\mathcal S_\ep = \{x\in \Om\cap B_\de(\o{x}_\ep): G_\ep(x) = G_\ep(y_\ep)\}.
\]
If $x \in \mathcal S_\ep$ then
\[
|x-\o{x}_\ep| = |y_\ep -\o{x}_\ep| e^{\frac{\beta_n}n (v_\ep(x) -v_\ep(y_\ep))},
\]
which implies the existence of a constant $c>0$ independent of $\ep$ such that
\[
e^{-c\de} Rr_\ep \leq |x-\o{x}_\ep| \leq e^{c\de} Rr_\ep.
\]
Consequently, we get
\[
\mathcal S_\ep \subset \Om\cap(B_{e^{c\de}Rr_\ep}(\o{x}_\ep) \setminus B_{e^{-c\de}Rr_\ep}(\o{x}_\ep)).
\]
By Lemmas \ref{convergevphiep} and \ref{boundedinSobolev}, we have
\begin{equation}\label{eq:lowbound*}
\inf_{\mathcal S_\ep } u_\ep\geq b_\ep = c_\ep + \frac{\vphi(e^{c\de}R) + o_\ep(R)}{c_\ep^{\frac1{n-1}}},
\end{equation}
and
\begin{equation}\label{eq:upbound*}
\sup_{\Om \cap \pa B_\de(\o{x}_\ep)} u_\ep \leq a_\ep = \frac{\sup_{\Om \cap \pa B_\de(\o{x}_\ep)} G + o_\ep(\de)}{c_\ep^{\frac1{n-1}}},
\end{equation}
where $o_\ep(R), o_\ep(\de)\to 0$ as $\ep \to 0$ and $R,\de$ are fixed, and $G$ is Green function \eqref{eq:Green}. For $\ep$ small enough, we have $a_\ep < b_\ep$. Denote $\mathcal G_\ep =\{x\in \Om \cap B_\de(\o x_\ep) : G_\ep(x) < G_\ep(y_\ep)\}$, and set $\o u_\ep = \min\{\max\{u_\ep,a_\ep\},b_\ep\}$. From \eqref{eq:lowbound*} and \eqref{eq:upbound*}, we get $\o u_\ep \in \Lam_\ep(-\frac n{\beta_n} \ln \de, G_\ep(y_\ep),a_\ep,b_\ep)$. By \eqref{eq:energy}, we obtain
\begin{equation}\label{eq:lowubar}
\lt(\int_{\mathcal G_\ep} |\na \o u_\ep|^n dx\rt)^{\frac1{n-1}} \geq \frac{(b_\ep -a_\ep)^{\frac n{n-1}}}{G_\ep(y_\ep) +\frac n{\beta_n} \ln \de}.
\end{equation}
Notice that 
\[
B_{e^{-c\de}Rr_\ep}(\o x_\ep) \cap \Om \subset \{G_\ep >G_\ep(y_\ep)\}.
\]
Using straightforward and tedious compuations, we get
\begin{align*}
\int_{\mathcal G_\ep} |\na \o u_\ep|^n dx &\leq \int_{\mathcal G_\ep} |\na u_\ep|^n dx\\
&\leq \int_{\Om \cap B_\de(\o x_\ep)} |\na u_\ep|^n dx -\int_{B_{e^{-c\de}Rr_\ep}(\o x_\ep)\cap \Om} |\na u_\ep|^n dx \\
&= 1 + \al \|u_\ep\|_n^n  -\int_{\Om\setminus B_\de(\o x_\ep)} |\na u_\ep|^n dx -\int_{B_{e^{-c\de}Rr_\ep}(\o x_\ep)\cap \Om} |\na u_\ep|^n dx \\
&= 1 + \frac1{c_\ep^{\frac n{n-1}}} \Bigg(\alpha \|G\|_n^n -\int_{\Om\setminus B_\de(\o x_\ep)} |\na G|^n dx +o_\ep(\de) + o_\ep(1)\Bigg)\\
&\qquad\qquad -\int_{B_{e^{-c\de}Rr_\ep}(\o x_\ep)\cap \Om} |\na u_\ep|^n dx.
\end{align*}
Integration by parts and \eqref{eq:Green} give
\begin{align*}
\int_{\Om\setminus B_\de(\o x_\ep)} |\na G|^n dx &= \int_{\Om\setminus B_\de(\o x_\ep)} (-\De_n G) G dx + \int_{\pa B_\de(\o x_\ep)\cap \Om} |\na G|^{n-2} \pa_\nu G G ds\\
&=\al \|G\|_n^n -\alpha \int_{B_{\de}(\o x_\ep)\cap \Om} |G|^n dx -\frac{\al \int_\Om |G|^{n-2} Gdx +1}{|\Om|} \int_{\Om \cap B_\de(\o x_\ep)} G dx\\
&\qquad + \int_{\pa B_\de(p)\cap \Om} |\na G|^{n-2} \pa_\nu G G ds + o_\ep(\de)\\
&=\al \|G\|_n^n -\frac n{\beta_n} \ln \de + A_p + o_\ep(\de) + o_\de(1).
\end{align*}
From the choice of the coordinate system $(V,\phi)$ and the fact ${y_\ep}_1/r_\ep \to 0$, we have
\begin{align*}
\int_{B_{e^{-c\de}Rr_\ep}(\o x_\ep)\cap \Om} |\na u_\ep|^n dx &= \int_{\phi(B_{e^{-c\de}Rr_\ep}(\o x_\ep))\cap \{y:y_1>0\}} |\na_g \tilde u_\ep|_g^n \sqrt{\text{\rm det}(g)} dy\\
&=(1+o_\ep(R))\int_{B_{(1+o_\ep(R))e^{-c\de}Rr_\ep}(y_\ep)\cap \{y:y_1>0\}} |\na \tilde u_\ep|^n dy\\
&=\frac{1}{c_\ep^{\frac n{n-1}}}\lt( \int_{B_{e^{-c\de}R}(0)\cap\{y:y_1>0\}} |\na \vphi|^n dx +o_\ep(R)\rt),
\end{align*}
and
\begin{align*}
\int_{B_{e^{-c\de}R}(0)\cap\{y:y_1>0\}} |\na \vphi|^n dx&=\frac n{\beta_n} \ln R + \frac{1}{\beta_n} \ln \frac{\om_{n-1}}{2n} \\
&\qquad\qquad + \frac{n-1}{\beta_n} \sum_{k=0}^{n-2} \frac{(-1)^{n-k-1}\binom{n-1}{k}}{n-k-1} + o_\de(1) + o_R(1)\\
&=\frac n{\beta_n} \ln R + \frac{1}{\beta_n} \ln \frac{\om_{n-1}}{2n} - \frac{n-1}{\beta_n} \sum_{k=1}^{n-1} \frac1k + o_\de(1) + o_R(1)
\end{align*}
Hence
\begin{align}\label{eq:11*}
\lt(\int_{\mathcal G_\ep} |\na \o u_\ep|^n dx\rt)^{\frac1{n-1}} &\leq 1 + \frac1{(n-1)c_\ep^{\frac n{n-1}}} \Bigg(\frac n{\beta_n} \ln \frac \de R - A_p -\frac{1}{\beta_n} \ln \frac{\om_{n-1}}{2n} \notag\\
&\qquad\qquad \qquad + \frac{n-1}{\beta_n} \sum_{k=1}^{n-1} \frac1k +o_\ep(\de) + o_\ep(1) + o_\de(1) + o_R(1)\Bigg),
\end{align}
For $\ep, \de$ sufficient small and $R$ sufficient large, here we use inequality $(1-t)^a \leq 1 -at$ for $0\leq t< 1$ and $0< a< 1$. From the expression of $a_\ep, b_\ep$, we have
\begin{align}\label{eq:22*}
(b_\ep -a_\ep)^{\frac n{n-1}} 
&\geq c_\ep^{\frac n{n-1}}\Bigg[1 + \frac1{c_\ep^{\frac n{n-1}}}\Bigg(\frac n{\beta_n} \ln \frac \de R -\frac{1}{\beta_n} \ln \frac{\om_{n-1}}{2n} -A_p +o_\ep(R)+o_\ep(\de) +o_\de(1)\Bigg)\Bigg]^{\frac n{n-1}}\notag\\
&\geq c_\ep^{\frac n{n-1}} + \frac n{n-1}\lt(\frac n{\beta_n} \ln \frac \de R -\frac{1}{\beta_n} \ln \frac{\om_{n-1}}{2n} -A_p +o_\ep(R)+ o_\ep(\de) +o_\de(1)\rt),
\end{align}
when $\ep, \de$ sufficient small and $R$ sufficient large, here we use inequality $(1-t)^a \geq 1 -at$ for $0\leq t< 1$ and $a >1$. By the choice of $y_\ep$, we have
\begin{equation}\label{eq:33*}
G_\ep(y_\ep) + \frac n{\beta_n} \ln \de =\frac{n}{\beta_n} \ln \frac\de{R} -\frac1{\beta_n} \ln \frac{\lam_\ep}{c_\ep^{\frac n{n-1}}} + \frac{\beta_\ep c_\ep^{\frac n{n-1}}}{\beta_n} + o_\de(1).
\end{equation}
Gathering \eqref{eq:lowubar}, \eqref{eq:11*}, \eqref{eq:22*} and \eqref{eq:33*} together, we get
\[
(1+o_\ep(1) + o_\ep(R) + o_\ep(\de)) \frac1{\beta_n} \ln \frac{\lam_\ep}{c_\ep^{\frac n{n-1}}} \leq \frac{1}{\beta_n} \ln \frac{\om_{n-1}}{2n} + A_p + \frac1{\beta_n} \sum_{k=1}^{n-1}\frac1k + o_\de(1) + o_{R}(1).
\] 
Let $\ep \to 0$, $\de\to 0$ and $R\to \infty$, we obtain
\[
\limsup_{\ep\to 0} \frac{\lam_\ep}{c_\ep^{\frac n{n-1}}} \leq \frac{\om_{n-1}}{2n} e^{\beta_n A_p + 1 + \frac12 + \cdots+ \frac1{n-1}}.
\]
This estimate and Lemma \ref{chantrensup} prove \eqref{eq:suptren}.
\end{proof}
\section{Proof of main theorems}
\begin{proof}[Proof of Theorem \ref{Main1}]
If $c_\ep$ is bounded, by applying elliptic estimates to \eqref{eq:EL}, we see that $u_\ep \to u^*$ in $C^1(\o \Om)$ for some function $u^* \in C^1(\o\Om)$ which implies Theorems \ref{Main1}. If $c_\ep \to \infty$ then Theorem \ref{Main1} follows from \eqref{eq:suptren}.
\end{proof}

\begin{proof}[Proof of Theorem \ref{Main2}]
We will construct a sequence $\phi_\ep \in\mathcal H$ such that $\|\na \phi_\ep\|_{n,\al} =1$ and
\begin{equation}\label{eq:suff}
\int_\Om e^{\beta |\phi_\ep|^{\frac n{n-1}}} dx > |\Om| + \frac{\om_{n-1}}{2n} e^{\beta_n A_p + 1 + \frac12 + \cdots + \frac1{n-1}},
\end{equation}
for $\ep >0$ small. Consequently, $c_\ep$ is bounded. Applying elliptic estimates to \eqref{eq:EL}, we get that $u_\ep \to u^*$ in $C^1(\o \Om)$ for some function $u^* \in C^1(\o \Om)$ which proves Theorem \ref{Main2}.

Denote $r =|x-p|$, notice that $G(x,p) = -\frac n{\beta_n} \ln r + A_p + \beta(x)$ with $\beta(x) = O(|x-p|)$. For $\ep >0$, denote $R = -\ln \ep$, consider the sequences of functions given by
\[
w_\ep =
\begin{cases}
c + \frac1{c^{\frac 1{n-1}}}\lt(-\frac{n-1}{\beta_n} \ln \lt(1+ \lt(\frac{\om_{n-1}}{2n}\rt)^{\frac1{n-1}} \frac{r^{\frac n{n-1}}}{\ep^{\frac n{n-1}}}\rt) + A\rt)&\mbox{if $0 < r < R\ep$,}\\
\frac1{c^{\frac 1{n-1}}}G &\mbox{if $r \geq R\ep$,}
\end{cases}
\]
and $\phi_\ep = w_\ep -\frac1{|\Om|}\int_\Om w_\ep dx$ where $\eta$ is cut-off function in $B_{2R\ep}(p)$, $\eta \equiv 1$ in $B_{R\ep}(p)$ and $\|\na \eta\|_\infty  = O((R\ep)^{-1})$, and $c,A$ are constants determined later.

In order to get $w_\ep \in H^1(\Om)$, we choose $A$ such that
\[
c + \frac1{c^{\frac 1{n-1}}}\lt(-\frac{n-1}{\beta_n} \ln \lt(1+ \lt(\frac{\om_{n-1}}{2n}\rt)^{\frac1{n-1}} R^{\frac n{n-1}}\rt)+A\rt) = \frac1{c^{\frac 1{n-1}}}\lt(-\frac n{\beta_n} \ln (R\ep) + A_p\rt),
\]
or 
\begin{equation}\label{eq:A}
A =-c^{\frac n{n-1}} + \frac{n-1}{\beta_n} \ln \lt(1+ \lt(\frac{\om_{n-1}}{2n}\rt)^{\frac1{n-1}} R^{\frac n{n-1}}\rt) -\frac n{\beta_n} \ln (R\ep) + A_p.
\end{equation}
We next compute some quantities concerning to $w_\ep$.
\begin{lemma}\label{eq:lemma5}
It holds
\begin{align}\label{eq:ab1}
\int_\Om |\na w_\ep|^n dx &= \frac1{c^{\frac n{n-1}}} \Bigg(\alpha \|G\|_2^2 -\frac n{\beta_n}\ln \ep +  A_p  + \frac{1}{\beta_n} \ln \frac{\om_{n-1}}{2n}\notag\\
&\qquad\qquad\quad - \frac{n-1}{\beta_n}\sum_{k=1}^{n -1} \frac1k +O(R^{-\frac n{n-1}}) +O(R\ep \ln(R\ep))\Bigg),
\end{align}
\begin{equation}\label{eq:meanwep}
c^{\frac 1{n-1}}\int_\Om w_\ep dx =O(R\ep(-\ln (R\ep))),
\end{equation}
and
\begin{equation}\label{eq:Lnnorm}
c^{\frac n{n-1}}\int_{\Om} |w_\ep|^n dx =\|G\|_n^n + O((R\ep)^n(-\ln (R\ep))^n).
\end{equation}
\end{lemma}
\begin{proof}
We first compute $\int_\Om |\na w_\ep|^n dx$ by splitting it as $\int_{\Om\cap B_{R\ep}(p)} + \int_{\Om \setminus B_{R\ep}(p)}$. A straightforward compuation shows that
\begin{align}\label{eq:aqw}
\int_{\Om \cap B_{R\ep}(p)} |\na w_\ep|^n dx & =\frac1{ c^{\frac n{n-1}}} \lt(\frac n{\beta_n}\rt)^n \lt(\frac{\om_{n-1}}{2n}\rt)^{\frac n{n-1}}\int_{\Om \cap B_{R\ep}(p)}\lt|\frac{\ep^{-\frac n{n-1}} r^{\frac1{n-1}}}{1+ \lt(\frac{\om_{n-1}}{2n}\rt)^{\frac1{n-1}} \frac{r^{\frac n{n-1}}}{\ep^{\frac n{n-1}}}}\rt|^n dx \notag\\
&= \frac1{ c^{\frac n{n-1}}} \lt(\frac n{\beta_n} \ln R + \frac{1}{\beta_n} \ln \frac{\om_{n-1}}{2n} - \frac{n-1}{\beta_n}\sum_{k=1}^{n -1} \frac1k +O(R^{-\frac n{n-1}})\rt),
\end{align}
and 
\begin{align*}
\int_{\Om\setminus B_{R\ep}(p)} |\na w_\ep|^n dx = \frac1{c^{\frac n{n-1}}}\int_{\Om\setminus B_{R\ep}(p)} |\na G|^ndx.
\end{align*}
Using integration by parts, \eqref{eq:Green} and the form of $G$ in \eqref{eq:Gform}, we get
\begin{equation}\label{eq:aqwx}
\int_{\Om\setminus B_{R\ep}(p)} |\na G|^n dx = \alpha \|G\|_n^n -\frac n{\beta_n}\ln (R\ep) +  A_p + O(R\ep \ln(R\ep)).
\end{equation}
\eqref{eq:aqw} and \eqref{eq:aqwx} prove \eqref{eq:ab1}.

By \eqref{eq:A}, we have
\begin{align*}
c^{\frac1{n-1}}w_\ep(x) &=\frac{n-1}{\beta_n} \lt(\ln \lt(1+ \lt(\frac{\om_{n-1}}{2n}\rt)^{\frac1{n-1}} R^{\frac n{n-1}}\rt) -\ln \lt(1+ \lt(\frac{\om_{n-1}}{2n}\rt)^{\frac1{n-1}} \frac{r^{\frac n{n-1}}}{\ep^{\frac n{n-1}}}\rt)\rt)\\
&\qquad -\frac n{\beta_n} \ln (R\ep) + A_p
\end{align*}
if $r =|x-p| < R\ep$. Hence 
\begin{equation}\label{eq:x123}
c^{\frac1{n-1}}|w_\ep(x)| \leq \frac{n-1}{\beta_n} \ln \lt(1+ \lt(\frac{\om_{n-1}}{2n}\rt)^{\frac1{n-1}} R^{\frac n{n-1}}\rt) - \frac n{\beta_n} \ln (R\ep) + |A_p|,
\end{equation}
if $r < R\ep$, and
\begin{equation}\label{eq:inte1}
c^{\frac1{n-1}} \int_{\Om \cap B_{R\ep}(p)} w_\ep dx = O((R\ep)^n (-\ln(R\ep))).
\end{equation}
Since $\int_\Om G dx =0$, we then have
\begin{align*}
c^{\frac1{n-1}} \int_{\Om \setminus B_{R\ep}(p)} w_\ep dx &= \int_{\Om \setminus B_{R\ep}(p)} G dx - \int_{\Om \cap (B_{2R\ep}(p) \setminus B_{R\ep}(p))} \eta \beta dx\\
&= -\int_{\Om \cap B_{R\ep}(p)} G dx - \int_{\Om \cap (B_{2R\ep}(p) \setminus B_{R\ep}(p))} \eta \beta dx.
\end{align*}
This equality together with the forms of $G$ and $\beta$ in \eqref{eq:Gform} implies
\begin{equation}\label{eq:inte2}
c^{\frac1{n-1}} \int_{\Om \setminus B_{R\ep}(p)} w_\ep dx =O((R\ep)^n (-\ln(R\ep))).
\end{equation}
\eqref{eq:meanwep} follows from \eqref{eq:inte1} and \eqref{eq:inte2}.

Finally, we have
\[
c^{\frac n{n-1}} \int_{\Om} |w_\ep|^n dx = \|G\|_n^n -\int_{\Om\cap B_{R\ep}(p)} |G|^n dx + \int_{\Om \cap B_{R\ep}}|c^{\frac1{n-1}} w_\ep|^n dx.
\]
This equality combining with \eqref{eq:Gform} and \eqref{eq:x123} implies \eqref{eq:Lnnorm}.
\end{proof}

From \eqref{eq:meanwep} and \eqref{eq:Lnnorm} we get
\begin{equation}\label{eq:bx1}
\|\phi_\ep\|_n^n=\frac{1}{c^{\frac n{n-1}}}\lt(\|G\|_n^n + O(R\ep \ln(R\ep))\rt).
\end{equation}
Therefore, we obtain by \eqref{eq:ab1} and \eqref{eq:bx1} that
\begin{align*}
\|\phi_\ep\|_{1,\alpha}^n & = \|\na w_\ep\|_n^n - \al \|\phi_\ep\|_n^n\\
&=\frac1{c^{\frac n{n-1}}} \Bigg(-\frac n{\beta_n}\ln \ep +  A_p  + \frac{1}{\beta_n} \ln \frac{\om_{n-1}}{2n}- \frac{n-1}{\beta_n}\sum_{k=1}^{n -1} \frac1k +O\lt(\frac1{(-\ln \ep)^{\frac n{n-1}}}\rt)\Bigg)
\end{align*}
here we use $R = -\ln \ep$. Hence, we can choose $c$ such that $\|\phi_\ep\|_{1,\al} =1$ for $\ep$ sufficient small, and 
\begin{equation}\label{eq:cvalue}
c^{\frac n{n-1}} = -\frac n{\beta_n}\ln \ep +  A_p  + \frac{1}{\beta_n} \ln \frac{\om_{n-1}}{2n}- \frac{n-1}{\beta_n}\sum_{k=1}^{n -1} \frac1k +O\lt(\frac1{(-\ln \ep)^{\frac n{n-1}}}\rt).
\end{equation}
We next compute $\int_\Om e^{\beta_n |\phi_\ep|^{\frac n{n-1}}} dx$. On $\Om\setminus B_{R\ep}(p)$ we have
\begin{align*}
\int_{\Om\setminus B_{R\ep}(p)} e^{\beta_n |\phi_\ep|^{\frac n{n-1}}} dx &\geq \int_{\Om\setminus B_{R\ep}(p)}\lt(1 + \frac{\beta_n^{n-1}}{(n-1)!} |\phi_\ep|^{n}\rt) dx\\
&= |\Om\setminus B_{R\ep}| + \frac{\beta_n^{n-1}}{(n-1)!} \frac{\|G\|_n^n}{c^{\frac n{n-1}}}+ O\lt(\frac1{(-\ln \ep)^{\frac n{n-1}}}\rt)\\
&=|\Om| + \frac{\beta_n^{n-1}}{(n-1)!} \frac{\|G\|_n^n}{c^{\frac n{n-1}}}+ O\lt(\frac1{(-\ln \ep)^{\frac n{n-1}}}\rt).
\end{align*}
On $\Om \cap B_{R\ep}(p)$, using the simple inequality $(1+t)^a \geq 1 +at$ for any $t> -1$ and $a>1$, and using \eqref{eq:meanwep} and \eqref{eq:cvalue} we have
\[
|\phi_\ep|^{\frac n{n-1}} \geq c^{\frac n{n-1}} +\frac{n}{n-1}\lt(-\frac{n-1}{\beta_n} \ln \lt(1+ \lt(\frac{\om_{n-1}}{2n}\rt)^{\frac1{n-1}} \frac{r^{\frac n{n-1}}}{\ep^{\frac n{n-1}}}\rt) + A\rt) + O\lt(\frac1{(-\ln \ep)^{\frac n{n-1}}}\rt).
\]
Hence
\begin{align*}
|\phi_\ep|^{\frac n{n-1}} &\geq -\frac 1{n-1}c^{\frac n{n-1}} +\frac n{n-1}(A+ c^{\frac n{n-1}})\\
&\hspace{2cm}-\frac n{\beta_n} \ln \lt(1+ \lt(\frac{\om_{n-1}}{2n}\rt)^{\frac1{n-1}} \frac{r^{\frac n{n-1}}}{\ep^{\frac n{n-1}}}\rt) + O\lt(\frac1{(-\ln \ep)^{\frac n{n-1}}}\rt)\\
&=-\frac n{\beta_n}\ln \ep +  A_p +\frac1{\beta_n}\sum_{k=1}^{n-1}\frac1k + \frac1{\beta_n}\ln \frac{\om_{n-1}}{2n} \\
&\hspace{2cm}-\frac n{\beta_n} \ln \lt(1+ \lt(\frac{\om_{n-1}}{2n}\rt)^{\frac1{n-1}} \frac{r^{\frac n{n-1}}}{\ep^{\frac n{n-1}}}\rt) + O\lt(\frac1{(-\ln \ep)^{\frac n{n-1}}}\rt).
\end{align*}
Integrating on $\Om\cap B_{R\ep}(p)$, we get
\begin{align*}
&\int_{\Om\cap B_{R\ep}(p)} e^{\beta_n|\phi_\ep|^{\frac n{n-1}}} dx\\
&\qquad\geq \frac{\om_{n-1}}{2n} e^{\beta_n A_p+\sum_{k=1}^{n-1}\frac1k}\ep^{-n} \int_{\Om\cap B_{R\ep}(p)} \lt(1+ \lt(\frac{\om_{n-1}}{2n}\rt)^{\frac1{n-1}} \frac{r^{\frac n{n-1}}}{\ep^{\frac n{n-1}}}\rt)^{-n} dx+ O\lt(\frac1{(-\ln \ep)^{\frac n{n-1}}}\rt)\\
&\qquad= \frac{\om_{n-1}}{2n} e^{\beta_n A_p+\sum_{k=1}^{n-1}\frac1k}\int_{B_R(0) \cap \frac{\Om-p}\ep} \lt(1+ \lt(\frac{\om_{n-1}}{2n}\rt)^{\frac1{n-1}} r^{\frac n{n-1}}\rt)^{-n} dx + O\lt(\frac1{(-\ln \ep)^{\frac n{n-1}}}\rt)\\
&\qquad= \frac{\om_{n-1}}{2n} e^{\beta_n A_p+\sum_{k=1}^{n-1}\frac1k} + O\lt(\frac1{(-\ln \ep)^{\frac n{n-1}}}\rt).
\end{align*}
Combining these estimates together and using \eqref{eq:cvalue}, we get
\begin{align*}
\int_\Om e^{\beta_n|\phi_\ep|^{\frac n{n-1}}} dx& \geq |\Om| + \frac{\om_{n-1}}{2n} e^{\beta_n A_p+\sum_{k=1}^{n-1}\frac1k} + \frac{\beta_n^{n-1}}{(n-1)!} \frac{\|G\|_n^n}{c^{\frac n{n-1}}}+ O\lt(\frac1{(-\ln \ep)^{\frac n{n-1}}}\rt)\\
&= |\Om| + \frac{\om_{n-1}}{2n} e^{\beta_n A_p+\sum_{k=1}^{n-1}\frac1k} + \frac{\beta_n^{n-1}}{(n-1)!c^{\frac n{n-1}}}\lt(\|G\|_n^n + O\lt(\frac1{(-\ln \ep)^{\frac1{n-1}}}\rt)\rt),
\end{align*}
Choosing $\ep >0$ sufficiently small, we see that \eqref{eq:suff} holds. This finishes the proof of Theorem \ref{Main2}.
\end{proof}

\section*{Acknowledgments}
This work was supported by CIMI's postdoctoral research fellowship.

\end{document}